\documentclass[12pt]{article}
\usepackage{amsmath,amssymb,amsthm}
\usepackage{hyperref}
\newcommand{\tmmathbf}[1]{\ensuremath{\boldsymbol{#1}}}

\newtheorem{theorem}{Theorem}[section]
\newtheorem{corollary}[theorem]{Corollary}
\newtheorem{lemma}[theorem]{Lemma}

\newtheorem{definition}[theorem]{Definition}

\numberwithin{equation}{section}

\usepackage{amssymb,latexsym,amsmath,epsfig,amssymb,amsthm, amsmath,color}
\def\V{\mathcal{V}}
\def\E{\mathcal{E}}
\def\A{\mathcal{A}}
\def\B{\mathcal{B}}
\begin{document}
\title{Distinct distances on regular varieties over finite fields}
\author{
    Pham Van Thang
    \and
    Do Duy Hieu
}
\date{}
\maketitle
\begin{abstract}
In this paper we study some generalized versions of a recent result due to Covert, Koh, and Pi (2015). More precisely, we prove that if a subset $\mathcal{E}$ in a \textit{regular} variety satisfies $|\mathcal{E}|\gg q^{\frac{d-1}{2}+\frac{1}{k-1}}$, then 
\[\Delta_{k, F}(\E):=\left\lbrace F(\mathbf{x}^1+\cdots+\mathbf{x}^k)\colon \mathbf{x}^i\in \E, 1\le i\le k\right\rbrace\supseteq \mathbb{F}_q\setminus\{0\},\]
for some certain families of polynomials $F(\mathbf{x})\in \mathbb{F}_q[x_1, \ldots, x_d]$.

\end{abstract}
\section{Introduction}
Let $\mathbb{F}_q$ be a finite field of order $q$, where $q$ is a prime power.  Let $D(\mathbf{x})=x_1^2+\cdots+x_d^2$ be a polynomial in $\mathbb{F}_q[x_1, \ldots, x_d]$. For $\E\subset\mathbb{F}_q^d$, we define the $D$-distance set of $\E$ to be
\[\Delta(\E)=\left\lbrace D(\mathbf{x}-\mathbf{y})\colon \mathbf{x}, \mathbf{y}\in \E\right\rbrace.\]

There are various papers studying the cardinality of $\Delta(\E)$, see for example \cite{bkt, ir, coko, copi, ko3} and references therein.  In this paper, we are interested in the case when $\E$ is a subset in a \textit{regular} variety. Let us first start with  a definition of regular varieties which is taken from \cite{copi}
\begin{definition}
For $\E\subseteq\mathbb{F}_q^d$,  let $\mathbf{1}_{\E}$ denote the characteristic function on $\E$. Let $F(\mathbf{x})\in \mathbb{F}_q[x_1, \ldots, x_d]$ be a polynomial. The variety $\mathcal{V}:=\{\mathbf{x}\in \mathbb{F}_q^d\colon F(\mathbf{x})=0\}$ is called a regular variety if $|\V| \asymp q^{d-1}$ and $\widehat{\mathbf{1}_\V(\mathbf{m})}\ll q^{-(d+1)/2}$ for all $\mathbf{m}\in \mathbb{F}_q^d\setminus \mathbf{0}$, where 
\[\widehat{\mathbf{1}_\V(\mathbf{m})}=\frac{1}{q^d}\sum_{\mathbf{x}\in \mathbb{F}_q^d}\chi(-\mathbf{m}\cdot\mathbf{x})\mathbf{1}_\V(\mathbf{x}).\]
\end{definition}
Here and throughout, $X \asymp Y$ means that there exist positive constants $C_1$ and $C_2$ such that $C_1Y< X <C_2Y$,  $X \ll Y$ means that there exists $C>0$ such that $X\le  CY$, and $X=o(Y)$ means that $X/Y\to 0$ as $q\to \infty$, where $X, Y$ are viewed as functions in $q$.

There are several examples of regular varieties as follows:
\begin{enumerate}
\item Spheres of nonzero radii:
\[S_j=\left\lbrace \mathbf{x}\in \mathbb{F}_q^d\colon ||\mathbf{x}||=j\right\rbrace, ~j\in \mathbb{F}_q^*:=\mathbb{F}_q\setminus \{0\} \quad  \cite{ir}\]
\item A paraboloid: 
\[P=\left\lbrace \mathbf{x}\in \mathbb{F}_q^d\colon x_1^2+\cdots+x_{d-1}^2=x_d\right\rbrace \quad \cite{tao}\]
\item Spheres defined by "Minkowski distance" with nonzero radii:
\[M_j=\left\lbrace \mathbf{x}\in \mathbb{F}_q^d\colon x_1\cdot x_2\cdots x_d=j\right\rbrace, ~~j\in \mathbb{F}_q^* \quad \cite{haha}.\]
\end{enumerate}
In 2007, Iosevich and Rudnev \cite{ir}, using Fourier analytic methods, made the first investigation on the distinct distance problem on the unit sphere in $\mathbb{F}_q^d$. More precisely, they proved the following.

\begin{theorem}[Iosevich et al., \cite{ir}]\label{018}
For $\E\subseteq S_1$ in $\mathbb{F}_q^d$ with $d\ge 3$. 
\begin{enumerate}
\item If $|\E|\ge Cq^{\frac{d}{2}}$ with a sufficiently large constant $C$, then there exists $c>0$ such that $|\Delta(\E)|\ge cq$.
\item If $d$ is even and $|\E|\ge Cq^{\frac{d}{2}}$ with a sufficiently large constant $C$, then $\Delta(\E)=\mathbb{F}_q$.
\item If $d$ is even, there exist $c>0$ and $\E\subset S_1$ such that $|\E|\ge cq^{\frac{d}{2}}$ and $\Delta(\E)\ne \mathbb{F}_q$.
\item If $d$ is odd and $|\E|\ge Cq^{\frac{d+1}{2}}$ with a sufficiently large constant $C>0$, then $\Delta(\E)=\mathbb{F}_q$. 
\item If $d$ is odd, there exist $c>0$ and $\E\subset S_1$ such that $|\E|\ge cq^{\frac{d+1}{2}}$ and $\Delta(\E)\ne \mathbb{F}_q$.
\end{enumerate}
\end{theorem}

Recently, Covert, Koh, and Pi \cite{copi} studied  a generalization of Theorem \ref{018}, namely they dealt with the  following question: How large does a subset $\E$ in a regular variety $\V$ need to be to make sure that  $\Delta_{k, D}(\E)=\mathbb{F}_q$ or $|\Delta_{k, D}(\E)|\gg q$, where 
\begin{equation}\label{0929}\Delta_{k, D}(\E):=\left\lbrace D(\mathbf{x}^1+\cdots+\mathbf{x}^k)\colon \mathbf{x}^i\in \E, 1\le i\le k\right\rbrace?\end{equation}

The main idea in the proof of Theorem \ref{018}  is to reduce the distance problem to the dot product problem since the distance between two points $\mathbf{x}$ and $\mathbf{y}$ in $S_1$ is $2-2\mathbf{x}\cdot \mathbf{y}$, where $\mathbf{x}\cdot \mathbf{y}=x_1y_1+\cdots +x_dy_d$. Therefore 
\begin{equation}\label{119}|\Delta(\E)|=|\Pi_{2}(\E)|:=\left\lbrace \mathbf{x}\cdot \mathbf{y}\colon \mathbf{x}, \mathbf{y}\in \E\right\rbrace.\end{equation}

For the case $k\ge 3$ and $\E\subset S_1$, one can check that 
\[|\Delta_{k, D}(\E)|=|\Pi_{k}(\E)|:=\left\vert \left\lbrace \sum_{i=1}^k\sum_{j=1}^ka_{ij}\cdot b_{ij}\cdot\mathbf{x}^i\cdot \mathbf{x}^j\colon \mathbf{x}^l\in \E, 1\le l\le k\right\rbrace\right\vert,\]
where $a_{ij}=1$ if $i<j$ and $0$ otherwise, and $b_{ij}=1$ for $i=1$ and $-1$ otherwise.  

However, it seems hard to get a good estimate on $|\Pi_{k}(\E)|$ when $k\ge 3$, and if the unit sphere $S_1$ is replaced by a general regular variety $\V$, there is no guarantee that the  equality (\ref{119}) will satisfy. Thus, in general, we can not apply the approach of the proof of Theorem \ref{018} to estimate the cardinality of $\Delta_{k, D}(\E)$.

Using a new approach with Fourier analytic techniques, Covert, Koh and Pi \cite{copi} established that the condition on the cardinality of $\E$ in Theorem \ref{018} can be improved to get $\Delta_{k, D}(\E)=\mathbb{F}_q$ with $k\ge 3$. The precise statement of their result is as follows.
\begin{theorem}[Covert et al., \cite{copi}]\label{028}
Suppose that $\V\subset \mathbb{F}_q^d$ is a regular variety, and assume that $k\ge 3$ is an integer and $\E\subseteq \V$. If $ q^{\frac{d-1}{2}+\frac{1}{k-1}}=o(|\E|)$, then  we have 
\[\Delta_{k, D}(\E)\supseteq \mathbb{F}_q^* ~~\mbox{for even $d\ge 2$},\]
and 
\[\Delta_{k, D}(\E)= \mathbb{F}_q ~~\mbox{for odd $d\ge 3$}.\]
\end{theorem}

It follows from Theorem \ref{018} that in order to get $\Delta_{2, D}(\E)=\mathbb{F}_q$, the sharp exponent of the sets $\E$ of $S_1$ must be $d/2$ for even $d\ge 4$, and $(d+1)/2$ for odd $d\ge 3$. Theorem \ref{028} implies that the exponent $d/2$ can be decreased to $\frac{d-1}{2}+\frac{1}{k-1}$ for $k\ge 3$ and any regular variety $\V\subseteq \mathbb{F}_q^d$.  

The main purpose of this note is to prove two generalizations of Theorem \ref{028} by employing techniques from spectral graph theory. Our first result is the following.
\begin{theorem}\label{thm1}
Let $Q$ be a non-degenerate quadratic form on $\mathbb{F}_q^d$. Suppose that $\V\subset \mathbb{F}_q^d$ is a regular variety, and assume that $k\ge 3$ is an integer and $\E\subseteq \V$. If $q^{\frac{d-1}{2}+\frac{1}{k-1}}=o(|\E|)$, then for any $t\in \mathbb{F}_q^*$ we have 

\[\left\vert \left\lbrace(\mathbf{x}_1, \ldots, \mathbf{x}_k)\in \E^k\colon Q(\mathbf{x}_1+\cdots+\mathbf{x}_k)=t\right\rbrace\right\vert =(1-o(1))\frac{|\E|^k}{q}.\]
\end{theorem}
\begin{corollary}\label{co2}
Let $Q$ be a non-degenerate quadratic form on $\mathbb{F}_q^d$. Suppose that $\V\subset \mathbb{F}_q^d$ is a regular variety, and assume that $k\ge 3$ is an integer and $\E\subseteq \V$. If $ q^{\frac{d-1}{2}+\frac{1}{k-1}}=o(|\E|)$, then  we have 
\[\Delta_{k, Q}(\E)\supseteq \mathbb{F}_q^*.\]
\end{corollary}
 Let $P(\tmmathbf{x})=\sum\limits_{j=1}^d a_jx_j^s$ with $s\geq 2, a_j\neq 0$ for all $j=1,\dots,d$ be a diagonal polynomial in $\mathbb F_q[x_1,\dots,x_d]$. We obtain the following generalization of Theorem  \ref{028}, which is inspired by the paper \cite{vinherdos}.
 \begin{theorem}\label{thm33}
 Suppose that $\V\subset \mathbb{F}_q^d$ is a regular variety, and assume that $k\ge 3$ is an integer and $\E\subseteq \V$. For $X\subseteq \mathbb{F}_q$,   if $|X||\E|^{2k-2}\gg q^{(d-1)(k-1)+2}$,  we have 
\[|X+\Delta_{k, P}(\E)|\gg q.\]
 \end{theorem}
 \begin{corollary}
  Suppose that $\V\subset \mathbb{F}_q^d$ is a regular variety, and assume that $k\ge 3$ is an integer and $\E\subseteq V$. If $|\E|\gg q^{\frac{d-1}{2}+\frac{1}{k-1}}$,  we have 
\[|\Delta_{k, P}(\E)|\gg q.\]
 \end{corollary}

The rest of this paper is organized as follows: In Sections $2$ and $3$, we construct some graphs which are main tools of our later proofs.  The proofs of Theorems \ref{thm1} and \ref{thm33} are presented in Sections $4$ and $5$, respectively. 
 \section{Pseudo-random graphs}
For a graph $G$ of order $n$, let $\lambda_1 \geq \lambda_2 \geq \ldots \geq \lambda_n$ be
the eigenvalues of its adjacency matrix. The quantity $\lambda (G) = \max
\{\lambda_2, - \lambda_n \}$ is called the second eigenvalue of $G$. A graph $G
= (V, E)$ is called an $(n, d, \lambda)$-graph if it is $d$-regular, has $n$
vertices, and the second eigenvalue of $G$ is at most $\lambda$. 

For two (not necessarily) disjoint subsets of vertices $U,
W \subseteq V$, let $e (U, W)$ be the number of ordered pairs $(u, w)$ such that
$u \in U$, $w \in W$, and $(u, w)$ is an edge of $G$. It is well known that if $\lambda$ is much smaller than the degree $d$, then $G$ has certain random-like properties.  More precisely, we have the following result on the number of edges between subsets in an $(n, d, \lambda)$-graph.
\begin{lemma}[Chapter 9, \cite{as}]\label{edge}
  Let $G = (V, E)$ be an $(n, d, \lambda)$-graph. For any two sets $B, C
  \subseteq V$, we have
  \[ \left| e (B, C) - \frac{d|B | |C|}{n} \right| \leq \lambda \sqrt{|B| |C|}. \]
\end{lemma}
In \cite{hanson}, Hanson et al. proved the following version of the expander mixing lemma on the number of edges between multi-sets of vertices in an $(n, d, \lambda)$-graph.
\begin{lemma}[\cite{hanson}]\label{expander}
Let $G=(V, E)$ be an $(n,d, \lambda)$-graph. The number of edges between two multi-sets of vertices $B$ and $C$ in $G$, which is denoted by $e(B, C)$, satisfies:
\[\left\vert e(B, C)-\frac{d|B||C|}{n}\right\vert\le \lambda\sqrt{\sum_{b\in B}m_B(b)^2}\sqrt{\sum_{c\in C}m_C(c)^2},\] where $m_X(x)$ is the multiplicity of $x$ in $X$.
\end{lemma}
\subsection{Finite Euclidean graphs}

Let $Q$ be a non-degenerate quadratic form on $\mathbb{F}_q^d$. For any $t \in \mathbb{F}_q$, the finite Euclidean
graph $E_q (d, Q, t)$ is defined as the graph with vertex set $\mathbb{F}_q^d$ and the edge
set

\begin{equation}
  E =\{(\tmmathbf{x}, \tmmathbf{y}) \in \mathbb{F}_q^d \times \mathbb{F}_q^d \, |\, \tmmathbf{x} \neq \tmmathbf{y}, \, Q (\tmmathbf{x} - \tmmathbf{y}) = t\}.
\end{equation}

The $(n, d, \lambda)$ form of the graph $E_{q}(d, Q, t)$ is estimated in the following theorem.
\begin{theorem}[Bannai et al. \cite{bst}, Kwok \cite{kwok}] \label{euclidean graphs}
  Let $Q$ be a non-degenerate quadratic form on $\mathbb{F}_q^d$. For any $t \in \mathbb{F}_q^{*}$, the graph $E_q (d, Q, t)$ is a $(q^d, (1+o(1))q^{d-1}, 2q^{(d-1)/2})$-graph.
\end{theorem} 
 
\section{Pseudo-random digraphs}
Let $G$ be a directed graph (digraph) on $n$ vertices where the in-degree and out-degree of each vertex are both $d$. 

We define the adjacency matrix of $G$, denoted by $A_G$,  as follows: $a_{ij} = 1$ if there is a directed edge from $i$ to $j$ and zero otherwise. Let $\lambda_1=d, \lambda_2, \ldots, \lambda_n$ be the eigenvalues of $A_G$.  These numbers are complex numbers, so we can not order them, but we have $|\lambda_i|\le d$ for all $1\le i \le n$. We define $\lambda(G):=\max_{|\lambda_i|\ne d}|\lambda_i|$.

An $n\times n$ matrix $A$ is normal if $A^tA = AA^t$, where $A^t$ is the transpose of $A$. We say that a digraph is normal if its adjacency matrix is a normal matrix. There is a simple way to check whether a digraph is normal. In a digraph $G$, let $N^+(x,y)$ be the set of vertices $z$ such that $\overrightarrow{xz}, \overrightarrow{yz}$ are edges, and $N^-(x,y)$ be the set of vertices $z$ such that $\overrightarrow{zx}, \overrightarrow{zy}$ are  edges. One can easily check that $G$ is normal if and only if $|N^+(x,y)| = |N^-(x,y)|$ for any two vertices $x$ and $y$. 

We say that  $G$ is an $(n, d, \lambda)$-digraph if $G$ is normal and $\lambda(G) \leq \lambda$.  Let $G$ be an $(n,d,\lambda)$-digraph. For two (not necessarily) disjoint subsets of vertices $U,
W \subset V$, let $e(U, W)$ be the number of ordered pairs $(u, w)$ such that
$u \in U$, $w \in W$, and $\overrightarrow{uw} \in E(G)$ (where $E(G)$ is the edge set of $G$). Vu \cite{van} developed a directed version of the Lemma \ref{edge} as follows.
\begin{lemma}[Vu, \cite{van}]\label{edge2}
  Let $G = (V, E)$ be an $(n, d, \lambda)$-digraph. For any two sets $B, C
  \subset V$, we have
  \[ \left| e(B, C) - \frac{d}{n}|B | |C| \right| \leq \lambda \sqrt{|B| |C|}. \]
\end{lemma}
By using  similar arguments as in the proofs of \cite[Lemma 16]{hanson} and \cite[Lemma 3.1]{van}, we obtain the multiplicity version of Lemma \ref{edge2}.
\begin{lemma}[Multiplicity version]\label{edge-c-sesb}
  Let $G = (V, E)$ be an $(n, d, \lambda)$-digraph. For any two multi-sets $B$  and $C$ of vertices , we have
  \[ \left| e(B, C) - \frac{d}{n}|B | |C| \right| \leq \lambda \sqrt{\sum_{b\in B}m_B(b)^2}\sqrt{\sum_{c\in C}m_C(c)^2}, \]
  where $m_X(x)$ is the multiplicity of $x$ in $X$.
\end{lemma}
We leave the proof of Lemma \ref{edge-c-sesb} to the interested reader.

\section{Proof of Theorem \ref{thm1}}
Let $H$ be a finite (additive) abelian group and $S$ be a subset of $H$. Define a directed Cayley graph $C_S$ as follows. The vertex of $C_S$ is $H$. There is a directed edge from $x$ to $y$ if and only if $y-x \in S$. It is clear that every vertex $C_S$ has out-degree $|S|$. Let $\chi_\alpha$, $\alpha \in H$, be the additive charaters of $H$. It is well known that for any $\alpha \in H$, $\sum_{s\in S} \chi_\alpha (s)$ is an eigenvalue of $C_S$, with respect the eigenvector $(\chi_\alpha(x))_{x\in H}$. 

Let $\V$ be a regular variety defined by 
\[\V:=\{\mathbf{x}\in \mathbb{F}_q^d\colon F(\mathbf{x})=0\},\]
for some polynomial $F\in \mathbb{F}_q[x_1, \ldots, x_d]$.

The Cayley graph $C_{\V}$ is defined with $H = \mathbb{F}_q^d$ and $S=\V$. In particular, the edge set of the Cayley graph $C_{\V}$ is defined by
\[E(C_{\V})=\{\overrightarrow{(\tmmathbf{x},\tmmathbf{y})} \in H \times H \colon \mathbf{y}-\mathbf{x}\in \V\}.\]

For any two vertices $\mathbf{x}$ and $\mathbf{y}$ in $H$, we have 
\[|N^+(\mathbf{x}, \mathbf{y})|=|N^-(\mathbf{x}, \mathbf{y})|=|(\mathbf{x}+\V)\cap (\mathbf{y}+\V)|,\]
which implies that $C_\V$ is normal. We now study the $(n, d, \lambda)$ form of this digraph in the next theorem.
\begin{theorem}\label{aaa}
The Cayley graph $C_\V$ is a $(q^d,|\V|, cq^{(d-1)/2})$-digraph for some positive constant $c$.
\end{theorem}
\begin{proof}
It is clear that $C_\V$ has $q^{d}$ vertices and the in-degree  and out-degree of each vertex are both $|\V|$. Next, we will estimate eigenvalues of $C_\V$. The exponentials (or characters of the additive group $\mathbb{F}_q^{d}$)
\begin{equation}
\chi_{\tmmathbf{m}}(\mathbf{x}) = \chi(\mathbf{x} \cdot \mathbf{m}), 
\end{equation}
for $\mathbf{x}, \mathbf{m} \in \mathbb{F}_q^d$, are eigenfunctions of the adjacency operator for the graph $C_{\V}$ corresponding to the eigenvalue
\begin{eqnarray*}
\lambda_{\mathbf{m}} & = & \sum_{\mathbf{x} \in \mathcal{V}} \chi_{\mathbf{m}}(\mathbf{x})\\
& = & \sum_{\mathbf{x}\in \V} \chi(\mathbf{x} \cdot \mathbf{m})\\
&=&q^d \widehat{\mathbf{1}_{\V}(-\mathbf{m})}\\
&\ll & q^{(d-1)/2},
\end{eqnarray*}
when $\mathbf{m}\ne \mathbf{0}$. If $\mathbf{m}=\mathbf{0}$, then $\lambda_{\mathbf{0}}=|\V|$, which is the
 largest eigenvalue of $C_\V$.  In other words,  $C_\V$ is a $(q^{d}, |\V|, cq^{(d-1)/2} )$-digraph for some positive constant $c$.
\end{proof}

In order to prove Theorem \ref{thm1}, we need the following notations.

For an even integer $k=2m \ge 2$ and $\E\subset \mathbb F_q^d,$   the $k$-energy is defined by
\[ \Lambda_k(\E)= \left\vert \left\lbrace (\mathbf{x}^1, \ldots, \mathbf{x}^k)\in \E^k\colon \mathbf{x}^1+\cdots+\mathbf{x}^m=\mathbf{x}^{m+1}+\cdots+\mathbf{x}^{k}\right\rbrace\right\vert.\]

For $\E\subseteq\mathbb{F}_q^d$, we define
\[\nu_k(t)=\left\vert \left\lbrace (\mathbf{x}^1, \ldots, \mathbf{x}^k)\in \E^k\colon Q(\mathbf{x}^1+\cdots+\mathbf{x}^k)=t\right\rbrace\right\vert.\]

In our next lemmas, we give  estimates on the magnitude of $\nu_k(t)$.
\begin{lemma}\label{lm872016}
For $\E\subset \mathbb{F}_q^d$ and $k\ge 2$ even, we have
\[\left\vert \nu_k(t)-(1+o(1))\frac{|\E|^k}{q}\right\vert\le q^{(d-1)/2}\Lambda_k(\E)\].
\end{lemma}
\begin{proof}
Suppose that $k=2m$. Let $\A$ and $\B$ be  multi-sets of points in $\mathbb{F}_q^d$ defined as follows
\[\A=\{\mathbf{x}_1+\cdots+\mathbf{x}_m\colon \mathbf{x}_i\in \E, 1\le i\le m\}, \quad \B=\{-\mathbf{x}_{m+1}-\cdots-\mathbf{x}_{k}\colon \mathbf{x}_i\in \E, m+1\le i\le k\}.\]
It is easy to check that 
\[\sum_{\mathbf{a}\in \A}m_\A(\mathbf{a})^2=\Lambda_k(\E), \quad \sum_{\mathbf{b}\in \B}m_\B(\mathbf{b})^2=\Lambda_k(\E),\]
and $\nu_k(t)$ is equal to the number of edges between $\A$ and $\B$ in the graph $E_q(d, Q, t)$. Thus the lemma follows immediately from Lemma \ref{expander} and Theorem \ref{euclidean graphs}.
\end{proof}

By using the same techniques, we get a similar result for the case $k$ odd.
\begin{lemma}\label{8720162}
For $\E\subset \mathbb{F}_q^d$ and $k\ge 3$ odd, we have
\[\left\vert \nu_k(t)-(1+o(1))\frac{|\E|^k}{q}\right\vert\le 2q^{(d-1)/2}\left(\Lambda_{k-1}(\E)\right)^{1/2}\left(\Lambda_{k+1}(\E)\right)^{1/2}.\]
\end{lemma}

Combining Lemmas \ref{lm872016} and \ref{8720162} leads to the following theorem.
\begin{theorem}\label{thmchinh2}
Let $\E$ be a set in $\mathbb{F}_q^d$. Then we have 
\begin{enumerate}
\item If $q^{\frac{d+1}{2}}\Lambda_k(\E)=o(|\E|^k)$ and $k$ is even, then 
\[\left\vert \left\lbrace (\mathbf{x}_1, \ldots, \mathbf{x}_k)\in \E^k\colon Q(\mathbf{x}_1+\cdots+\mathbf{x}_k)=t\right\rbrace\right\vert =(1+o(1))\frac{|\E|^k}{q}.\]

\item If $q^{\frac{d+1}{2}}(\Lambda_{k-1}(\E))^{1/2}(\Lambda_{k+1}(\E))^{1/2}=o(|\E|^k)$ and $k$ is odd, then 
\[\left\vert \left\lbrace (\mathbf{x}_1, \ldots, \mathbf{x}_k)\in \E^k\colon Q(\mathbf{x}_1+\cdots+\mathbf{x}_k)=t\right\rbrace\right\vert =(1+o(1))\frac{|\E|^k}{q}.\]
\end{enumerate}
\end{theorem}

Theorem \ref{thmchinh2} implies that in order to prove Theorem \ref{thm1}, it is sufficient to bound $\Lambda_k(\E)$.
\begin{lemma}\label{hai}
For a regular variety $\V\subset \mathbb{F}_q^d$. If $k\ge 4$ is even, and $\E\subset \V$, we have 
\[\left\vert \Lambda_k(\E)- (1+o(1))\frac{|\E|^{k-1}}{q}\right\vert\ll q^{(d-1)/2}(\Lambda_{k-2}(\E))^{1/2}(\Lambda_{k}(\E))^{1/2}.\]
\end{lemma}    
\begin{proof}
Since $\E$ is a subset in the variety $\V$, we have the following estimate
\[ \Lambda_k(\E)\le \sum_{\mathbf{x}_1, \ldots, \mathbf{x}_{k-1}\in \E}\mathbf{1}_\V(\mathbf{x}_1+\cdots+\mathbf{x}_{k/2}-\mathbf{x}_{k/2+1}-\cdots-\mathbf{x}_{k-1}).\]

Let  $\A$ and $\B$ be two multi-sets defined by
\[\A:=\{\mathbf{x}_1+\cdots+\mathbf{x}_{k/2}\colon \mathbf{x}_i\in \E, 1\le i\le k/2\},\] and \[\B:=\{-\mathbf{x}_{k/2+1}-\cdots-\mathbf{x}_{k-1}\colon \mathbf{x}_i\in \E, k/2+1\le i\le k-1\}.\]
It is clear that 
\[\sum_{\mathbf{a}\in \A}m_\A\mathbf(a)^2=\Lambda_k(\E), \quad \sum_{\mathbf{b}\in \B}m_\B(\mathbf{b})^2=\Lambda_{k-2}(\E).\]
On the other hand, $\Lambda_k(\E)$ is equal to the number of edges between $\A$ and $\B$ in the Cayley graph $C_\V$. Thus the lemma follows from Lemmas \ref{edge-c-sesb} and \ref{aaa}.
\end{proof}
For $\E\subseteq \V$ and $k\ge 4$ even, it follows from Lemma \ref{hai} that 
\[\Lambda_k(\E)\ll \frac{|\E|^{k-1}}{q}+q^{(d-1)/2}(\Lambda_{k-2}(\E))^{1/2}(\Lambda_{k}(\E))^{1/2}.\]
Solving this inequality in terms of $\Lambda_k(\E)$ gives us 

\[\Lambda_k(\E)\ll q^{d-1}\Lambda_{k-2}(\E)+\frac{|\E|^{k-1}}{q}.\]
Using inductive arguments, we obtain the following estimate for $\E\subseteq \V$ and $k\ge 4$ even

\begin{equation}\label{317}\Lambda_k(\E)\ll q^{\frac{(d-1)(k-2)}{2}}\Lambda_2(\E)+\frac{|\E|^{k-1}}{q}\sum_{j=0}^{(k-4)/2}\left(\frac{q^{d-1}}{|\E|^2}\right)^j.\end{equation}

If we assume that $|\E|>q^{(d-1)/2}$, then the inequality (\ref{317}) implies the following theorem.
\begin{theorem}\label{coco872016}
Let $\E$ be a subset of a regular variety $\V$ in $\mathbb{F}_q^d$ with $|\E|>q^{(d-1)/2}$. 
\begin{enumerate}
\item If $k\ge 2$ is even, then 
\[\Lambda_k(\E)\ll q^{\frac{(d-1)(k-2)}{2}}|\E|+\frac{|\E|^{k-1}}{q}.\]
\item If $k\ge 3$ is odd, then 
\[\Lambda_{k-1}(\E)\Lambda_{k+1}(\E)\ll q^{(d-1)(k-2)}|\E|^2+q^{\frac{(d-1)(k-3)-2}{2}}|\E|^{k+1}+\frac{|\E|^{2k-2}}{q^2}.\]
\end{enumerate}
\end{theorem}
Note that the first statement of Theorem \ref{coco872016} follows from (\ref{317}) with
the facts that $\Lambda_2(\E)=|\E|$ and $\frac{q^{d-1}}{|\E|^2}<1$, and the second is a consequence of the first one.

We are now ready to prove Theorem \ref{thm1}.

\begin{proof}[Proof of Theorem \ref{thm1}]
We now consider two following cases:

{\bf Case 1:} If $k\ge 2$ is even and $q^{\frac{d-1}{2}+\frac{1}{k-1}}=o(|\E|)$, then it follows from Theorem \ref{coco872016} that \[q^{\frac{d+1}{2}}\Lambda_k(\E)=o(|\E|^k).\]

{\bf Case 2:} If $k\ge 3$ is odd and $q^{\frac{d-1}{2}+\frac{1}{k-1}}=o(|\E|)$, then it follows from Theorem \ref{coco872016} that \[q^{\frac{d+1}{2}}(\Lambda_{k-1}(\E))^{1/2}(\Lambda_{k+1}(\E))^{1/2}=o(|\E|^k).\]

In other words, Theorem \ref{thm1} follows immediately from Theorem \ref{thmchinh2}.

\end{proof}

\section{Proof of Theorem \ref{thm33}}
To prove Theorem \ref{thm33}, we need to construct  a new Cayley graph as follows.

 Let $P(\mathbf{x})=\sum\limits_{j=1}^d a_jx_j^s \in \mathbb{F}_q[x_1,\dots,x_d]$ with $s\geq 2, a_j\neq 0$ for all $j=1,\dots,d$, and  \[P'(x_1,\ldots, x_{2d})=P(x_1, \ldots, x_d)-P(x_{d+1}, \ldots, x_{2d})\in \mathbb{F}_q[x_1, \ldots, x_{2d}].\] We define the graph $C_{P'}(\mathbb{F}_q^{2d+1})$ to be the Cayley graph with $H=\mathbb{F}_q\times \mathbb{F}_q^{2d}$ and $S=\lbrace (x_0,\mathbf{x})\in \mathbb{F}_q\times \mathbb{F}_q^{2d}~|~x_0+P'(\mathbf{x})=0\rbrace$, i.e. \[E(C_{P'}(\mathbb{F}_q^{2d+1}))=\left\lbrace\overrightarrow{((x_0,\mathbf{x}), (y_0,\mathbf{y}))}\in H\times H \colon y_0-x_0+P'(\mathbf{y}-\mathbf{x})=0\right\rbrace.\]

The $(n, d, \lambda)$ form of $C_{P'}(\mathbb{F}_q^{2d+1})$ was studied in \cite{vinherdos}.
\begin{lemma}[\cite{vinherdos}]\label{bodechinh}
For any odd prime power $q$, $d\ge 1$, then $C_{P'}(\mathbb{F}_q^{2d+1})$ is a 
  \[(q^{2d+1}, q^{2d}, q^{d})-\mbox{digraph}.\]
\end{lemma} 

For $\E\subseteq\mathbb{F}_q^d$ and $X\subseteq\mathbb{F}_q$,  define
\[\nu_{P, k}(t)=\left\vert \{(a, \mathbf{x}_1, \ldots, \mathbf{x}_k)\in X\times \E^k\colon a+P(\mathbf{x}_1+\cdots+\mathbf{x}_k)=t\}\right\vert.\]

Our next lemmas are the main steps in the proof of Theorem \ref{thm33}.
\begin{lemma}\label{3171}
For $\E\subseteq \mathbb{F}_q^d$ and $k\ge 2$ even, we have the following estimate
\[\sum_{t\in \mathbb{F}_q}\nu_{P, k}(t)^2\le \frac{|\E|^{2k}|X|^2}{q}+q^{d}|X|\Lambda_{k}(\E)^2.\]
\end{lemma}
\begin{proof}
Let $\A$ and $\B$ be multi-sets defined by:
\[ \A:=\{(a, -\mathbf{x}_1-\cdots-\mathbf{x}_{k/2}, -\mathbf{y}_1-\cdots-\mathbf{y}_{k/2})\colon a\in X, \mathbf{x}_i, \mathbf{y}_i\in \E\},\]
and 
\[ \B:=\{(b, \mathbf{x}_{k/2+1}+\cdots+\mathbf{x}_{k}, \mathbf{y}_{k/2+1}+\cdots+\mathbf{y}_{k/2+1})\colon b\in X, \mathbf{x}_i, \mathbf{y}_i\in \E\}.\]

One can check that 
\[\sum_{\mathbf{x}\in \A}m_\A(\mathbf{x})^2=|X|\Lambda_k(\E)^2, ~\sum_{\mathbf{x}\in \B}m_\B(\mathbf{x})^2=|X|\Lambda_k(\E)^2, ~ |\A|=|\B|=|X||\E|^k.\]
On the other hand, it is clear that $\sum_{t\in \mathbb{F}_q}\nu_{P, k}^2$ is equal to the number of edges from $\A$ to $\B$ in the graph $C_{P'}(\mathbb{F}_q^{2d+1})$. Thus it follows from Lemma \ref{edge-c-sesb} and Theorem \ref{bodechinh} that 
 \[\sum_{t\in \mathbb{F}_q}\nu_{P, k}(t)^2\le \frac{|\E|^{2k}|X|^2}{q}+q^{d}|X|\Lambda_{k}(\E)^2.\]
 This ends the proof of the lemma.
\end{proof}

By employing the same techniques, we get a similar result for the case $k\ge 3$ odd.
\begin{lemma}\label{3172}
For $\E\subseteq \mathbb{F}_q^d$ and $k\ge 3$ odd, we have the following estimate
\[\sum_{t\in \mathbb{F}_q}\nu_{P, k}(t)^2\le \frac{|\E|^{2k}|X|^2}{q}+q^{d}|X|\Lambda_{k-1}(\E)\Lambda_{k+1}(\E).\]
\end{lemma}

We are now ready to prove Theorem \ref{thm33}.
\begin{proof}[Proof of Theorem \ref{thm33}]
It follows from the proof of Theorem $2.6$ in \cite{vinherdos} that 
\[|X+\Delta_{k, P}(\E)|\gg \frac{|X|^2|\E|^{2k}}{\sum_{t\in \mathbb{F}_q}\nu_{P, k}(t)^2}.\]

Therefore from Lemma \ref{3171} and Lemma \ref{3172}, we get two following cases:

\begin{enumerate}
\item If $k\ge 2$ is even, we obtain
\[|X+\Delta_{k, P}(\E)|\gg \min\left\lbrace  \frac{|X||\E|^{2k}}{q^d\Lambda_k(\E)^2}, q\right\rbrace.\]

\item If $k\ge 3$ is odd, we obtain
\[|X+\Delta_{k, P}(\E)|\gg \min\left\lbrace  \frac{|X||\E|^{2k}}{q^d\Lambda_k(\E)\Lambda_{k-1}(\E)}, q\right\rbrace.\]

\end{enumerate}

Thus Theorem \ref{thm33} follows immediately from Theorem \ref{coco872016}, which concludes the proof of the theorem.
\end{proof}

\section*{Acknowledgements.}
\thispagestyle{empty}
The first author was partially supported by Swiss National Science Foundation grants 200020-162884 and 200020-144531.

\vspace{1cc}
\hfill\\
Department of Mathematics,\\
EPF Lausanne\\
Switzerland\\
E-mail: thang.pham@epfl.ch\\
\bigskip
\hfill\\
Institute of Mathematics, \\
Vietnam Academy of Science and Technology\\  
E-mail: ddhieu@math.ac.vn

\end{document}